\documentclass[a4paper,10pt]{amsart}

\usepackage{amsmath,amsfonts,amssymb}
\usepackage[utf8]{inputenc}
\usepackage[english]{babel}
\usepackage[T1]{fontenc}
\usepackage[all]{xy}
\usepackage{textcomp}
\usepackage{color}
\usepackage{xcolor}

\usepackage[hypertexnames=false,
backref=page,
    pdftex,
    pdfpagemode=UseNone,
    breaklinks=true,
    extension=pdf,
    colorlinks=true,
    linkcolor=blue,
    citecolor=blue,
    urlcolor=blue,
]{hyperref}

\allowdisplaybreaks[1]

\newcommand{\referenza}{}

\newtheorem{prop}{Proposition}[section]
\newtheorem*{prop*}{Proposition \referenza}
\newtheorem{thm}[prop]{Theorem}
\newtheorem*{thm*}{Theorem \referenza}
\newtheorem{cor}[prop]{Corollary}
\newtheorem*{cor*}{Corollary \referenza}

\theoremstyle{definition}

\newtheorem{ex}[prop]{Example}

\newcommand{\Z}{\mathbb{Z}}
\newcommand{\Q}{\mathbb{Q}}
\newcommand{\R}{\mathbb{R}}
\newcommand{\C}{\mathbb{C}}

\newcommand{\g}{\mathfrak{g}}
\newcommand{\sspace}{\text{\--}}
\newcommand{\ssspace}{\text{\textdblhyphen}}

\newcommand{\scalar}[2]{\left\langle #1 \,\middle|\, #2 \right\rangle}

\DeclareMathOperator{\im}{i}
\DeclareMathOperator{\imm}{im}
\DeclareMathOperator{\de}{d}

\DeclareMathOperator{\End}{End}

\newcommand{\del}{\partial}
\newcommand{\delbar}{\overline{\del}}

\title[Generalized complex cohomologies]{Cohomologies of generalized complex manifolds and nilmanifolds}

\author{Daniele Angella}
\address[Daniele Angella]{Centro di Ricerca Matematica ``Ennio de Giorgi''\\
Collegio Puteano, Scuola Normale Superiore\\
Piazza dei Cavalieri 3\\
56126 Pisa, Italy}
\email{daniele.angella@sns.it}
\email{daniele.angella@gmail.com}

\author{Simone Calamai}
\address[Simone Calamai]{Dipartimento di Matematica e Informatica ``Ulisse Dini''\\
Università di Firenze\\
via Morgagni 67/A, 50134\\
Firenze, Italy}
\email{scalamai@math.unifi.it}

\author{Hisashi Kasuya}
\address[Hisashi Kasuya]{Department of Mathematics\\
Tokyo Institute of Technology\\
1-12- 1-H-7, O-okayama, Meguro\\
Tokyo 152-8551, Japan}
\email{khsc@ms.u-tokyo.ac.jp}
\email{kasuya@math.titech.ac.jp}

\thanks{During the preparation of the work, the first author has been granted by a research fellowship by Istituto Nazionale di Alta Matematica INdAM and by a Junior Visiting Position at Centro di Ricerca ``Ennio de Giorgi''; he is also supported by the Project PRIN ``Varietà reali e complesse: geometria, topologia e analisi armonica'', by the Project FIRB ``Geometria Differenziale e Teoria Geometrica delle Funzioni'', by SNS GR14 grant ``Geometry of non-Kähler manifolds'', and by GNSAGA of INdAM. The second author is supported by GNSAGA of INdAM, by the Project PRIN ``Varietà reali e complesse: geometria, topologia e analisi armonica'', and by SNS GR14 grant ``Geometry of non-Kähler manifolds''. The third author is  supported by JSPS Research Fellowships for Young Scientists.}

\begin{document}

\begin{abstract}
 We study generalized complex cohomologies of generalized complex structures constructed from certain symplectic fibre bundles over complex manifolds. We apply our results in the case of left-invariant generalized complex structures on nilmanifolds and to their space of small deformations.
\end{abstract}

\maketitle

\section*{Introduction}

Generalized complex geometry, in the sense of N. Hitchin, M. Gualtieri, and G.~R. Cavalcanti, \cite{hitchin, gualtieri-annals, cavalcanti-phd}, unifies symplectic and complex geometries in a unitary framework. In such a way, it clarifies the parallelism between results for (non-K\"ahler) complex manifolds and for (non-K\"ahler) symplectic manifolds.

\medskip

We recall that a generalized complex structure on a differentiable manifold $M$ is an endomorphism ${\mathcal J}\in {\rm End}(TM\oplus T^*M)$ such that ${\mathcal J}^{2}=-1$ and the $i$-eigenbundle $L\subset (TM\oplus T^*M)\otimes \C$ is involutive with respect to the Courant bracket \eqref{eq:courant}. If $\omega\in\wedge^2M$ (see as an isomorphism $TM\to T^*M$) is a symplectic structure, respectively $J\in{\rm End}(TM)$ is a complex structure on $M$, then
$$ \mathcal{J}_\omega \;:=\;
\left(
\begin{array}{c|c}
 0 & -\omega^{-1} \\
\hline
 \omega & 0
\end{array}
\right) \;,
\qquad
\text{ respectively }
\qquad
\mathcal{J}_J
\;:=\;
\left(
\begin{array}{c|c}
 -J & 0 \\
\hline
 0 & J^*
\end{array}
\right)
$$
are generalized complex structures on $M$.
In view of the generalized Darboux theorem \cite[Theorem 3.6]{gualtieri-annals} proved by M. Gualtieri, these examples constitute the basic models of generalized complex structures near regular points.

A generalized complex structure $\mathcal{J}$ on $M$ of dimension $2n$ yields a decomposition of complex differential forms $\wedge ^{\bullet}  T^*M\otimes \C = \bigoplus_{j=-n}^{n}U^{j}$, whence the bi-differential $\Z$-graded complex
$$ \left( {\mathcal U}^{\bullet},\, \del,\, \delbar \right) \;. $$
In this note, we are interested in the generalized Dolbeault cohomologies
$$
 GH^{\bullet}_{\del}(M) \;:=\; \frac{\del}{\imm\del}
 \qquad \text{ and } \qquad
 GH^{\bullet}_{\delbar}(M) \;:=\; \frac{\ker\delbar}{\imm\delbar} \;,
$$
and in the generalized Bott-Chern and Aeppli cohomologies
$$
 GH^{\bullet}_{BC}(M) \;:=\; \frac{\ker\del\cap\ker\delbar}{\imm\del\delbar}
 \qquad \text{ and } \qquad
 GH^{\bullet}_{A}(M) \;:=\; \frac{\ker\del\delbar}{\imm\del+\imm\delbar} \;.
$$
(Note that, in the complex case, the generalized $\del$ and $\delbar$ operators coincide with the complex operators, and so, up to a change of graduation, the above cohomologies are exactly the Dolbeault and the Bott-Chern cohomologies. In the symplectic case, the generalized Dolbeault cohomology is isomorphic to the de Rham cohomology, and the generalized Bott-Chern cohomology has been studied by L.-S. Tseng and S.-T. Yau, see \cite{tseng-yau-1, tseng-yau-2, tseng-yau-3, tsai-tseng-yau}.)

More precisely, look at the $i$-eigenbundle $L\subset (TM\oplus T^*M)\otimes \C$ of $\mathcal{J}\in\End((TM\oplus T^*M) \otimes \C)$ with the Lie algebroid structure given by the Courant bracket and the projection $\pi\colon L\to TM\otimes \C$. Take a generalized holomorphic bundle, that is, a complex vector bundle $E$ with  a Lie algebroid connection 
\[\delbar\colon  \mathcal{C}^{\infty}(\wedge^{k} L^{\ast}\otimes E)\to  \mathcal{C}^{\infty}(\wedge^{k+1} L^{\ast}\otimes E)
\] 
satisfying $\delbar\circ \delbar=0$. Consider
\[
GH_{\delbar}^{n-\bullet}(M,E) \;:=\; H^{\bullet}(L,E) \;:=\; \frac{\ker\left(\delbar\colon \mathcal{C}^{\infty}(\wedge^{\bullet} L^{\ast}\otimes E) \to \mathcal{C}^{\infty}(\wedge^{\bullet+1} L^{\ast}\otimes E)\right)}{\imm\left(\delbar\colon \mathcal{C}^{\infty}(\wedge^{\bullet-1} L^{\ast}\otimes E) \to \mathcal{C}^{\infty}(\wedge^{\bullet} L^{\ast}\otimes E)\right)} \;.
\]

\medskip

One way to construct generalized complex structures on manifolds is the following.
Let $p\colon X\to B$ be a symplectic fibre bundle with a generic fibre $(F,\sigma)$. Assume that the base $B$ is a compact complex manifold and that there is a closed form $\omega$ on the total space $X$ which restricts to the symplectic form $\sigma$ on the generic $F$.
Then we can construct a non-degenerate pure form, and then a generalized complex structure on $E$.

We construct the following Leray spectral sequence for computing the generalized cohomology of such an $E$.

\renewcommand{\referenza}{\ref{spect-dol}}
\begin{cor*}
Let $p\colon X\to B$ a symplectic fibre bundle with a generic fibre $(F,\sigma)$ of dimension $2\ell$
such that:
\begin{itemize}
\item $B$ is a compact complex manifold of complex dimension $k$;
\item we have a closed form $\omega$ on the total space $X$ which restricts to the symplectic form $\sigma$ on the generic $F$.
\end{itemize}
Consider the generalized complex structure $\mathcal J$ on $X$ defined by $\omega$ and the complex structure of $B$
and the $i$-eigenbundle $L$ of $\mathcal J$.
Let $W$ be a complex vector bundle over $X$
such that $W=p^{\ast}W^{\prime}$ for a holomorphic vector bundle $W^{\prime}$ over the complex manifold $B$.
We regard $W$ as a generalized holomorphic bundle.
Consider the flat vector bundle ${\bf H}(F)=\bigcup_{x\in B}H^{\bullet}(F_{b})$ over $B$.

Then there exists a spectral sequence $\left\{ E_{r}^{\bullet,\bullet} \right\}_{r}$ which converges to $GH_{\delbar}^{k+\ell-\bullet}(X)$
such that 
\[E^{p,q}_{2}\cong GH_{\delbar}^{k-p}(B,{\bf H}^{\ell-q}(F))\;.
\]
\end{cor*}

\medskip

As an application of the above results, we investigate generalized cohomologies of nilmanifolds $M=\left.\Gamma\middle\backslash G\right.$, that is, compact quotients of connected simply-connected nilpotent Lie groups $G$. We consider left-invariant generalized complex structures on $M$, equivalently, linear generalized complex structures on the Lie algebra $\g$ of $G$.
Note that left-invariant generalized complex structures on nilmanifolds are generalized Calabi-Yau, that is, the canonical line bundle $K$ is trivial; whence $GH^{n-\bullet}_{\delbar}(M)=H^{\bullet}(L)$.

In this context, we have a generalized complex decomposition also at the level of the Lie algebra, namely, $\wedge^\bullet\g^*=\bigoplus_j\mathfrak{U}^j$, and a (finite dimensional) bi-differential $\Z$-graded sub-complex
$$ \left( \mathfrak{U}^\bullet,\, \del,\, \delbar \right) \to \left( \mathcal{U}^\bullet,\, \del,\, \delbar \right) \;. $$
It induces the map $GH^\bullet_{\delbar}(\g) \to GH^\bullet_{\delbar}(M)$ in cohomology, which is in fact always injective.

\renewcommand{\referenza}{\ref{cor:lieiso}}
\begin{cor*}
Let $G$ be a connected simply-connected nilpotent Lie group and $\g$ the Lie algebra of $G$.
We suppose that $G$ admits a lattice $\Gamma$ and consider the $\Q$-structure $\g_{\Q}\subset \g$ induced by $\Gamma$.
We assume that there exists an ideal $\mathfrak h\subset \g$ so that:
\begin{enumerate}
\item $\g_{\Q}\cap \mathfrak h$ is a $\Q$-structure of $\mathfrak h$;
\item $\g/{\mathfrak h}$ admits a complex structure $J$;
\item we have a closed $2$-form $\omega\in \wedge^{2} \g^{\ast}$
yielding $\omega\in \wedge^{2}{\mathfrak h}^{\ast}$ non-degenerate form on $\mathfrak h$;
\item $\iota\colon \wedge^{\bullet,\bullet}(\g/\mathfrak h)^{\ast}\otimes\C\to \wedge^{\bullet,\bullet}\Gamma H\backslash G$ induces an isomorphism on the Dolbeault cohomology.
\end{enumerate}
Then the inclusion $\iota\colon({\mathfrak U}^{\bullet},\del,\delbar)\to ({\mathcal U}^{\bullet},\del,\delbar) $ induces isomorphisms
$GH_{\delbar}(\g)\cong GH_{\delbar}(\Gamma\backslash G)$, and
$GH_{\del}(\g)\cong GH_{\del}(\Gamma\backslash G)$, and 
$GH_{BC}(\g)\cong GH_{BC}(\Gamma\backslash G)$.
\end{cor*}

As regards the fourth assumption, we note that it holds, e.g., when $J$ is either bi-invariant, or holomorphically-parallelizable, or Abelian, or rational, or nilpotent, see \cite{console-survey} and the references therein.

As an explicit example, we study a generalized complex structure on the Kodaira-Thurston manifold in Section \ref{sec:kt}.

The above invariance result for generalized cohomologies is stable under small deformations.

\renewcommand{\referenza}{\ref{thm:coh-def}}
\begin{thm*}
Let $\Gamma\backslash G$ be a nilmanifold with a left-invariant generalized complex structure $\mathcal J$; denote by $\g$ be the Lie algebra of $G$.
If the isomorphism $GH_{\delbar}(\g)\cong GH_{\delbar}(\Gamma\backslash G)$ holds on the original generalized complex structure $\mathcal J$,
then the same isomorphism holds on the deformed generalized complex structure $\mathcal J_{\epsilon(t)}$
for sufficiently small $t$.
\end{thm*}
For complex case, theorems of this type are found in \cite{console-fino, angella-1, angella-kasuya-2}.

\medskip

Finally, we apply the above result on nilmanifolds to study their space of small deformations. In particular, we prove that any small deformation of a generalized complex structure on a nilmanifold with invariant generalized cohomology is (equivalent to) a left-invariant structure.

\renewcommand{\referenza}{\ref{thm:def-inv}}
\begin{thm*}
Let $\Gamma\backslash G$ be a nilmanifold with a left-invariant generalized complex structure $\mathcal J$; denote by $\g$ be the Lie algebra of $G$.
If the isomorphism $GH_{\delbar}(\g)\cong GH_{\delbar}(\Gamma\backslash G)$ holds on the original generalized complex structure $\mathcal J$,
then any sufficiently small deformation of generalized complex structure is equivalent to a left-invariant complex structure $\mathcal J_{\epsilon}$ with $\epsilon \in \wedge^{2} \mathfrak L^{\ast}$ satisfying the Maurer-Cartan equation.
\end{thm*}

This result is a generalization of \cite[Theorem 2.6]{Rollenske}.

\bigskip

\noindent{\sl Acknowledgments.}
The authors would like to thank Gil R. Cavalcanti, S\"onke Rollenske, Giovanni Bazzoni, and Adela Latorre for useful discussions and for helpful comments on preliminary versions.

\section{Generalized complex structures}
Let $M$ be a compact differentiable manifold  of dimension $2n$.
Consider the vector bundle $TM\oplus T^*M$, endowed with the natural symmetric pairing
$$
\scalar{X+\xi}{Y+\eta} \;:=\; \frac{1}{2}\, \left( \xi(Y)+\eta(X) \right) \;.
$$
We define the action of $TM\oplus T^*M$ on $\wedge ^{\bullet}  T^*M$ so that 
$$
(X+\xi)\cdot \rho=i_{X}\rho+\xi\wedge \rho
$$
We define the \emph{Courant bracket} on
the space $\mathcal{C}^{\infty}\left(TX\oplus T^*X\right)$   such that 
\begin{equation}\label{eq:courant}
\left[X+\xi,\, Y+\eta\right] \;:=\; \left[X,\, Y\right] + \mathcal{L}_X\eta - \mathcal{L}_Y\xi - \frac{1}{2}\, \de \left(\iota_X\eta-\iota_Y\xi\right) \;.
\end{equation}
A \emph{generalized complex structure} on $M$ is an endomorphism ${\mathcal J}\in {\rm End}(TM\oplus T^*M)$ such that  ${\mathcal J}^{2}=-1$ and the $i$-eigenbundle $L\subset (TM\oplus T^*M)\otimes \C$ involutive with respect to the Courant bracket.

A form $\rho$ in $\wedge^{\bullet}  T^*M\otimes \C$ is called  \emph{pure } 
 if it can be written as
\[\rho =e^{B+i\omega}\Omega
\]
where $B, \omega \in \wedge ^{2}T^*M$ and $\Omega=\theta_{1}\wedge \dots \wedge \theta_{k}$ with $\theta_{1},\dots, \theta_{k}\in T^*M\otimes \C$.
A pure form  $\rho\in \wedge ^{\bullet} T^*M\otimes \C$  is non-degenerate if 
\[\omega^{n-k}\wedge \Omega\wedge \overline{\Omega}\not=0.
\]
For a generalized complex structure ${\mathcal J}$ with the $i$-eigenbundle $L$, we have the \emph{canonical line bundle } 
 $K\subset \wedge ^{\bullet}  T^*M\otimes \C$   such that
$$
L={\rm Ann}(K)=\left\{v \in (TM\oplus T^*M)\otimes \C \;\middle\vert\; v\cdot K=0\right\} \;.
$$ 
 Any $\rho \in K$ is a non-degenerate pure form and  any $\phi\in \mathcal{C}^{\infty}(K)$ is integrable, i.e., there exists $v\in \mathcal{C}^{\infty}\left(TX\oplus T^*X\right)$ satisfying
\[d\phi=v\cdot \phi.
\]
Conversely, if we have a line bundle  $K\subset \wedge ^{\bullet}  T^*M\otimes \C$ so that  any $\rho \in K$ is a non-degenerate pure form and  any $\phi\in \mathcal{C}^{\infty}(K)$ is integrable, then we have a generalized complex structure whose $i$-eigenbundle is $L={\rm Ann}(K)$.

\medskip

For a generalized complex manifold $(M,\mathcal J)$ with the 
$i$-eigenbundle  $L\subset (TM\oplus T^*M)\otimes \C$ and 
 the canonical line bundle   $K\subset \wedge ^{\bullet}  T^*M\otimes \C$,
for $j\in\Z$, we  define
$$ U^j \;:=\; \wedge^{n-j}\bar L \cdot K \;\subseteq\; \wedge^\bullet X \otimes \C \;. $$
Then we have 
\[\wedge ^{\bullet}  T^*M\otimes \C= \bigoplus_{j=-n}^{n}U^{j}\;.
\]
Denote ${\mathcal U}^{j}=\mathcal{C}^{\infty}(U^{j})$.
Then, by the integrability, we have $d{\mathcal U}^{j}\subset {\mathcal U}^{j-1}\oplus {\mathcal U}^{j+1}$. 
We consider the decomposition  $d=\del+\delbar$ such that 
$$ \del\colon {\mathcal U}^{j}\to  {\mathcal U}^{j+1} \qquad \text{ and } \qquad \delbar\colon {\mathcal U}^{j}\to  {\mathcal U}^{j-1} \;.$$
Hence we have the bi-differential $\Z$-graded complexes $({\mathcal U}^{\bullet},\del,\delbar)$.

We define the generalized Dolbeault cohomologies
\begin{eqnarray*}
 GH^{\bullet}_{\del}(M) &:=& \frac{\ker\left(\del\colon \mathcal U^\bullet \to \mathcal U^{\bullet+1}\right)}{\imm\left(\del\colon \mathcal  U^{\bullet-1} \to \mathcal U^{\bullet}\right)}, \\
 GH^{\bullet}_{\delbar}(M) &:=& \frac{\ker\left(\delbar\colon \mathcal U^\bullet \to \mathcal U^{\bullet-1}\right)}{\imm\left(\delbar\colon \mathcal  U^{\bullet+1} \to \mathcal U^{\bullet}\right)} \;.
\end{eqnarray*}
Define also the generalized Bott-Chern and Aeppli cohomologies
\begin{eqnarray*}
GH^{\bullet}_{BC}(M) &:=& \frac{\ker\left(\del\colon\mathcal  U^{\bullet} \to \mathcal  U^{\bullet+1} \right) \cap \ker\left(\delbar\colon\mathcal U^{\bullet} \to\mathcal U^{\bullet-1} \right)}{\imm \left(\del\delbar \colon\mathcal U^{\bullet} \to\mathcal U^{\bullet} \right)} \;, \\[5pt]
GH^{\bullet}_{A}(M) &:=& \frac{\ker \left( \del\delbar \colon \mathcal U^{\bullet} \to \mathcal U^{\bullet} \right) }{\imm \left(\del \colon\mathcal U^{\bullet-1} \to \mathcal U^{\bullet} \right) + \imm \left( \delbar \colon\mathcal U^{\bullet+1} \to\mathcal U^{\bullet} \right) } \;.
\end{eqnarray*}

\medskip

A {\em generalized Hermitian metric} on a generalized complex manifold $(M,\mathcal J)$ is a self-adjoint orthogonal transformation $\mathcal G\in {\rm End}(TM\oplus TM^{\ast})$ such that $\langle {\mathcal G}v,v\rangle>0$ for $v\not=0$
and $\mathcal J\mathcal G=\mathcal G\mathcal J$.
For a generalized Hermitian metric $\mathcal G$, we can define the generalized Hodge star operator $\star:\mathcal U^{\bullet} \to \mathcal U^{\bullet} $ (see \cite[Section 3]{cavalcanti-jgp}) and its conjugation  $\bar\star$.
Define $\delbar^{\ast}=-\bar\star\delbar \bar\star$ and $\Delta_{\delbar}=\delbar\delbar^{\ast}+\delbar^{\ast}\delbar$.
Then $\Delta_{\delbar}$ is an elliptic operator and every cohomology class $\alpha\in GH_{\delbar}^{\bullet}(M)$ admits a unique representative $a\in \ker \Delta_{\delbar}$.

\medskip

It is known that the vector bundle $L$ with the Courant bracket and the projection $\pi: L\to TM\otimes \C$ is a Lie algebroid.
By this, we have the differential graded algebra structure on $\mathcal{C}^{\infty}(\wedge^{\bullet} L^{\ast})$ with the differential $d_{L}: \mathcal{C}^{\infty}(\wedge^{k} L^{\ast})\to  \mathcal{C}^{\infty}(\wedge^{k+1} L^{\ast})$
such that 
\begin{multline*}d_{L}\omega (v_{1},\dots, v_{k+1}) \;=\; \sum_{i<j} (-1)^{i+j-1}\omega([v_{i},v_{j}],v_{1},\dots ,\hat{v}_{i},\dots, \hat{v}_{j},\dots, v_{k+1})\\
+\sum _{i=1}^{k+1}(-1)^{i}\pi(v_{i})\left(\omega(v_{1},\dots,\hat{v}_{i},\dots,v_{k+1})\right) \;.
\end{multline*}
It is known that $\left(\mathcal{C}^{\infty}(\wedge^{\bullet} L^{\ast}), d_{L}\right)$ is a elliptic complex (see \cite[Proposition 3.12]{gualtieri-annals}).
A {\em generalized holomorphic bundle} is a complex vector bundle $E$ with  a Lie algebroid connection 
\[\delbar\colon  \mathcal{C}^{\infty}(\wedge^{k} L^{\ast}\otimes E)\to  \mathcal{C}^{\infty}(\wedge^{k+1} L^{\ast}\otimes E)
\] 
satisfying $\delbar\circ \delbar=0$.
For a generalized holomorphic bundle $(E,\delbar)$,
we define the Lie algebroid cohomology
\[H^{\bullet}(L,E)= \frac{\ker\left(\delbar\colon \mathcal{C}^{\infty}(\wedge^{\bullet} L^{\ast}\otimes E) \to \mathcal{C}^{\infty}(\wedge^{\bullet+1} L^{\ast}\otimes E)\right)}{\imm\left(\delbar\colon \mathcal{C}^{\infty}(\wedge^{\bullet-1} L^{\ast}\otimes E) \to \mathcal{C}^{\infty}(\wedge^{\bullet} L^{\ast}\otimes E)\right)}
\]

Identifying $L^{\ast}=\overline{L}$ by the pairing,
$\delbar: {\mathcal U}^{n-k}\to  {\mathcal U}^{n-k-1}$ can be viewed as a Lie algebroid connection
\[\delbar:  \mathcal{C}^{\infty}(\wedge^{k} L^{\ast}\otimes K)\to  \mathcal{C}^{\infty}(\wedge^{k+1} L^{\ast}\otimes K) \;.
\] 
such that
\[\delbar(\omega \otimes s)=d_{L}\omega \otimes s+(-1)^{k}\omega \otimes d s.
\]
Hence the canonical line bundle $K$ is generalized holomorphic and we have  $GH^{n-\bullet}_{\delbar}(M)=H^{\bullet}(L,K)$.
For a generalized holomorphic bundle $(E,\delbar)$,
we denote $GH^{n-\bullet}_{\delbar}(M,E)=H^{\bullet}(L,K\otimes E)$.

If there exists a nowhere-vanishing closed section $\rho\in \mathcal{C}^{\infty}(K)$,
we call $\mathcal J$ a {\em generalized Calabi-Yau structure}.
In this case, we have $GH^{n-\bullet}_{\delbar}(M)=H^{\bullet}(L)$.

\medskip

By the identification $L^{\ast}=\overline{L}$,
we can define the Schouten bracket $[\sspace,\ssspace]$ on $\mathcal{C}^{\infty}(\wedge^{\bullet} L^{\ast})$.
For sufficiently small $\epsilon\in \mathcal{C}^{\infty}(\wedge^{2} L^{\ast})$,
we obtain the small deformation of the isotropic subspace
\[L_{\epsilon}=(1+\epsilon)L\subset (TM\oplus T^*M)\otimes \C.
\]
Consider the endomorphism ${\mathcal J}_{\epsilon}\in {\rm End}(TM\oplus T^*M)$ whose   $i$-eigenbundle and $-i$-eigenbundle are $L_{\epsilon}$ and $\overline{L_{\epsilon} }$ respectively.
Then ${\mathcal J}_{\epsilon}$ is a generalized complex structure if and only if $\epsilon$ satisfies the Maurer-Cartan equation:
\[d_{L}\epsilon+\frac{1}{2}[\epsilon,\epsilon]=0.
\]

As similar to Complex Geometry, we can apply the Kuranishi theory.
Choose a Hermitian metric on $L$.
Consider the adjoint operator $d_{L}^{\ast}$, the Laplacian operator $\Delta_{L}=d_{L}d_{L}^{\ast}+d_{L}^{\ast}d_{L}$, the projection $H: \mathcal{C}^{\infty}(\wedge^{\bullet} L^{\ast})\to \ker \Delta_{L}$ and the Green operator $G$ (i.e., the operator on $\mathcal{C}^{\infty}(\wedge^{\bullet} L^{\ast})$ so that $G\Delta_{L}+H={\rm id}$).
Let $\epsilon_{1}\in  \ker \Delta_{L}\cap  \mathcal{C}^{2}(\wedge^{\bullet} L^{\ast})$.
We consider the formal power series $\epsilon(\epsilon_{1})$ with values in $\mathcal{C}^{\infty}(\wedge^{\bullet} L^{\ast})$ given inductively by
 \[\epsilon_{r}(\epsilon_{1})=\frac{1}{2}\sum_{s=1}^{r-1} d_{L}^{\ast}G[\epsilon_{s}(\epsilon_{1}), \epsilon_{r-s}(\epsilon_{1})].\]
Then, for sufficiently small $\epsilon_{1}$, the formal power series $\epsilon(\epsilon_{1})$  converges.
\begin{thm}[{\cite[Theorem 5.5]{gualtieri-annals}}]\label{defku}
Any sufficiently  small deformation of the generalized complex structure $\mathcal J$ is equivalent to a generalized complex structure ${\mathcal J}_{\epsilon(\epsilon_{1})}$ for some $\epsilon_{1}\in  \ker \Delta_{L}\cap  \mathcal{C}^{2}(\wedge^{\bullet} L^{\ast})$ such that $\epsilon(\epsilon_{1})$ satisfies the Maurer-Cartan equation.
\end{thm}

\begin{ex}
Let $M$ be a compact $2n$-dimensional manifold endowed with a symplectic structure $\omega \in \wedge^2 M$. Consider the induced isomorphism $\omega \colon TM \to T^*M$.
The symplectic structure gives rise to the generalized complex structure
$$ \mathcal{J}_\omega \;:=\;
\left(
\begin{array}{c|c}
 0 & -\omega^{-1} \\
\hline
 \omega & 0
\end{array}
\right) \;. $$

In this case, we obtain the $i$-eigenbundle
$$L=\{X-i\omega(X): X\in TM\otimes \C\},
$$
the canonical line bundle $K=\langle e^{i\omega} \rangle$
and
$$ U^{n-\bullet} \;=\; \Phi\left(\wedge^\bullet X \otimes \C\right) \;, $$
where
$$ \Phi(\alpha) \;:=\; \exp{\left(\im\omega\right)}\, \left(\exp{\left(\frac{\Lambda}{2\im}\right)}\, \alpha\right) \;, $$
and $\Lambda := -\iota_{\omega^{-1}}$.
In particular, we have the Lie algebroid isomorphism $TM\otimes \C\cong L$, $\mathcal J$ is generalized Calabi-Yau and hence we have an isomorphism $H^{\ast}(M)\cong H^{\ast}(L)\cong  GH^{n-\bullet}_{\delbar}(M)$.
Moreover, 
we have \cite[Corollary 1]{cavalcanti-jgp},
$$ \Phi\de \;=\; \delbar\Phi  \quad\text{ and }\quad \Phi\de^{\Lambda} \;=\; 2\im\del\Phi \;, $$
where $\de^\Lambda:=\left[\de,\Lambda\right]$
and this implies that 
$GH^{k}_{BC}\left(X\right)$ and $GH^{k}_{A}\left(X\right)$ are isomorphic to the symplectic Bott-Chern and Aeppli cohomologies introduced and studied by L.-S. Tseng and S.-T. Yau, see \cite{tseng-yau-1, tseng-yau-2, tseng-yau-3, tsai-tseng-yau}.
\end{ex}

\begin{ex}
Let $M$ be a compact $2n$-dimensional manifold endowed with a complex structure $J\in\End(TM)$. The complex structure induces the generalized complex structure

where $J^*\in\End(T^*M)$ denotes the dual endomorphism of $J\in\End(TX)$.
In this case, we obtain the $i$-eigenbundle $L=T^{0,1}M\oplus T^{\ast1,0}M$,
the canonical line bundle $K=\wedge^{n} T^{\ast1,0}M$ and
$$\mathcal  U^\bullet \;=\; \bigoplus_{p-q=\bullet}\wedge^{p,q}X \;, $$
with the differentials
$$ \del \;=\; \del_J \qquad \text{ and } \qquad \delbar \;=\; \delbar_J \;, $$
where $\del_J $ and $\delbar_J$ are the usual Dolbeault operators on a complex manifold.
The Lie algebroid complex $\mathcal{C}^{\infty}(\wedge ^{\bullet}L^{\ast})$ is 
 $\mathcal{C}^{\infty}(\wedge ^{\bullet}(T^{1,0}M\oplus T^{\ast0,1}M) )$ with the differential $d_{L}$ which is the usual Dolbeault operator.
\end{ex}

\section{Fibrations and spectral sequences}\label{fisp}

A \emph{symplectic fibre bundle} is a smooth fibre bundle $p\colon X\to B$ so that the fibre $F$ is a compact symplectic manifold and
the structural group is the group of symplectomorphisms.
Let $p\colon X\to B$ be a symplectic fibre bundle with a generic fibre $(F,\sigma)$
such that:
\begin{itemize}
\item $B$ is a compact complex manifold of complex dimension $k$;
\item we have a closed form $\omega$ on the total space $X$ which restricts to the symplectic form $\sigma$ on the generic $F$.
\end{itemize}
Taking a local trivialization $U\times F\subset X$, for a local holomorphic coordinates set $(z_{1},\dots,z_{k})$ in $U$
we obtain a non-degenerate pure form
\[\rho=e^{i\omega}dz_{1}\wedge\dots \wedge dz_{k}
\]
and it gives a generalized complex structure on $E$ whose $i$-eigenbundle $L$ is given by
\[L_{\vert U}=T^{0,1}U\oplus T^{\ast1,0}U\oplus \{X-i\omega(X): X\in TF\otimes \C\}.
\]
We consider the sub-bundle $S$ so that $ S_{\vert U}=\{X-i\omega(X): X\in TF\otimes \C\}\subset L_{\vert U}$.
Then, $S$ is  involutive with respect to the Courant bracket.

\medskip

For $b\in B$ and $F_{b}=p^{-1}(b)$, denoting by $H^{\bullet}(F_{b})$ the $\C$-valued de Rham cohomology of $F_{b}$, we consider the vector bundle ${\bf H}(F)=\bigcup_{x\in B}H^{\bullet}(F_{b})$.
Then  ${\bf H}(F)$ is a flat vector bundle over $B$.
Hence, in this case,  ${\bf H}(F)$ is a holomorphic vector bundle over the complex manifold $B$.

Consider the  bundle ${\mathcal F}=TF_{b}\otimes\C$ of the vectors tangent to the fibres.
Then ${\mathcal F}$ is a Lie algebroid.
Consider the Lie algebroid cohomology $H^{\ast}({\mathcal F})$ then we have an isomorphism 
\[H^{\ast}({\mathcal F})\cong \mathcal{C}^{\infty}({\bf H}(F))
\]
see \cite[Chapter I.2.4]{hattori}.

\medskip

Let $(W,\delbar)$ be a generalized holomorphic bundle over $X$.
Define the subspace $F^{p}\mathcal{C}^{\infty}(\wedge^{\bullet} L^{\ast}\otimes W)\subset \mathcal{C}^{\infty}(\wedge^{\bullet} L^{\ast}\otimes W)$
so that
\begin{multline*}
F^{p}\mathcal{C}^{\infty}(\wedge^{p+q} L^{\ast}\otimes W)\\=
\left\{\phi\in \mathcal{C}^{\infty}(\wedge^{p+q} L^{\ast}\otimes W) \;\middle\vert\; \phi(X_{1},\dots, X_{p+q})=0 \text{ for } X_{\ell_{1}},\dots, X_{\ell_{q+1}}\in S \right\}.
\end{multline*}
Then $F^{p}\mathcal{C}^{\infty}(\wedge^{\bullet} L^{\ast}\otimes W)$ is a decreasing bounded filtration of $(\mathcal{C}^{\infty}(\wedge^{\bullet} L^{\ast}\otimes W),\delbar)$.
Hence we obtain the spectral sequence $\left\{ E_{r}^{\bullet,\bullet} \right\}_{r}$ for  this filtration.

We suppose that $W=p^{\ast}W^{\prime}$ for a holomorphic vector bundle $W^{\prime}$ over the complex manifold $B$.
For a local holomorphic coordinates set $(z_{1},\dots, z_{k})$ of $B$, locally we have
\[E^{p,q}_{0}\cong \wedge^{p} \left\langle d\bar{z}_{1},\dots, d\bar{z}_{k}, \frac{\partial}{\partial z_{1}},\dots  \frac{\partial}{\partial z_{k}} \right\rangle \otimes_{\mathcal{C}^{\infty}(B)} W^{\prime} \otimes_{\mathcal{C}^{\infty}(B)} \mathcal{C}^{\infty}(\wedge^{q} S^{\ast})
\]
with the differential
$$ d_{0}={\rm id}\otimes d_{S} $$
where $d_{S}$ is the differential on the Lie algebroid complex $\mathcal{C}^{\infty}(\wedge^{q} S^{\ast})$.
By using the $\omega$, we have a Lie algebroid isomorphism ${\mathcal F}\ni X\mapsto X-i\omega(X)\in S$.
Hence we obtain
\[E^{p,q}_{1}\cong \wedge^{p} \left\langle d\bar{z}_{1},\dots, d\bar{z}_{k}, \frac{\partial}{\partial z_{1}},\dots  \frac{\partial}{\partial z_{k}} \right\rangle \otimes_{\mathcal{C}^{\infty}(B)} W^{\prime} \otimes_{\mathcal{C}^{\infty}(B)} {\bf H}^{\bullet}(F)
\]
with the differential
$$d_{1}=\delbar_{B}$$
where $\delbar_{B}$ is the usual Dolbeault operator on the complex manifold $B$.
Thus, globally, we obtain
\[E^{p,q}_{1}\cong \mathcal{C}^{\infty}(\wedge^{p}L_{B}^{\ast}\otimes W^{\prime}\otimes {\bf H}^{q}(F))
\]
with the differential $d_{1}=\delbar$ which is the Lie algebroid connection on the holomorphic bundle $W^{\prime}\otimes {\bf H}(F)$
where $L_{B}=T^{0,1}B\oplus T^{\ast1,0}B$.
Hence we have
\[E^{p,q}_{2}\cong H^{p}(L_{B}, W^{\prime}\otimes {\bf H}^{q}(F)).
\]

We have shown the following result.
\begin{thm}\label{spect}
Let $p\colon X\to B$ a symplectic fibre bundle with a generic fibre $(F,\sigma)$
such that:
\begin{itemize}
\item $B$ is a compact complex manifold of complex dimension $k$;
\item we have a closed form $\omega$ on the total space $X$ which restricts to the symplectic form $\sigma$ on the generic $F$.
\end{itemize}
Consider the generalized complex structure $\mathcal J$ on $X$ defined by $\omega$ and the complex structure of $B$
and the $i$-eigenbundle $L$ of $\mathcal J$.
Let $W$ be a complex vector bundle over $X$
such that $W=p^{\ast}W^{\prime}$ for a holomorphic vector bundle $W^{\prime}$ over the complex manifold $B$.
We regard $W$ as a generalized holomorphic bundle.

Then there exists a spectral sequence $\left\{ E_{r}^{\bullet,\bullet} \right\}_{r}$ which converges to $H^{\bullet}(L, W)$
such that 
\[E^{p,q}_{2}\cong H^{p}(L_{B}, W^{\prime}\otimes {\bf H}^{q}(F)).
\]
\end{thm}

Set $W=K$ which is the canonical line bundle of $(X,\mathcal J)$.
Then as a bundle, we have $p^{\ast}K_{B}\cong K$
where $K_{B}$ is the canonical line bundle of the complex manifold $B$.
Hence we have:
\begin{cor}\label{spect-dol}
Consider the same setting in Theorem \ref{spect}.
Suppose $\dim B=2k$, $\dim F=2\ell$. 

Then there exists a spectral sequence $\left\{ E_{r}^{\bullet,\bullet} \right\}_{r}$ which converges to $GH_{\delbar}^{k+\ell-\bullet}(X)$
such that 
\[E^{p,q}_{2}\cong GH_{\delbar}^{k-p}(B,{\bf H}^{\ell-q}(F))\;.
\]
\end{cor}

\section{Generalized complex structures on Lie algebras}
Let $\g$ be a $2n$-dimensional Lie algebra.
We consider the Lie algebra ${\mathbb D}\g=\g\oplus \g^{\ast}$ with the bracket
\[[X+\zeta, Y+\eta]=[X,Y]+\mathcal{L}_{X}\eta-\mathcal{L}_{Y}\zeta
\]
for $X, Y\in\g$ and $\zeta, \eta\in \g^{\ast}$.
A generalized complex structure on $\g$ is a complex structure on  ${\mathbb D}\g$ which is orthogonal with respect to the pairing
\[\langle X+\zeta, Y+\eta\rangle=\frac{1}{2}(\zeta(Y)+\eta(X)).
\]

Consider the complex $\wedge ^{\bullet}\g^{\ast}_{\C}$ of the Lie algebra $\g_\C:=\g\otimes_\R\C$.
A form $\rho\in \wedge ^{\bullet}\g^{\ast}_{\C}$ is a pure form of type $k$ if it can be written as
\[\rho =e^{B+i\omega}\Omega
\]
where $B, \omega \in \wedge ^{2}\g^{\ast}$ and $\Omega=\theta_{1}\wedge \dots \wedge \theta_{k}$ with $\theta_{1},\dots, \theta_{k}\in \wedge ^{1}\g^{\ast}_{\C}$.
A pure form $\rho\in \wedge ^{\bullet}\g^{\ast}_{\C}$ of type $k$ is non-degenerate if 
\[\omega^{n-k}\wedge \Omega\wedge \overline{\Omega}\not=0.
\]
A pure form  $\rho\in \wedge ^{\bullet}\g^{\ast}_{\C}$ of type $k$ is integrable if there exists $X+\zeta\in {\mathbb D}\g$
such that 
\[d\rho=(X+\zeta)\cdot\rho.
\]

\begin{thm}[\cite{cavalcanti-gualtieri}]\label{nilcal}
If $\g$ is nilpotent, then any  non-degenerate integrable pure form is closed.
\end{thm}

For a non-degenerate integrable pure form $\rho\in \wedge ^{\bullet}\g^{\ast}_{\C}$ of type $k$, 
we have the sub Lie algebra ${\mathfrak L}\subset {\mathbb D}\g_{\C}$ such that
\[{\mathfrak L}={\rm Ann}(\rho)= \{X+\zeta\in {\mathbb D}\g_{\C}\vert  (X+\zeta)\cdot\rho=0\}.
\]
We have the decomposition ${\mathbb D}\g_{\C}={\mathfrak L}\oplus \overline{\mathfrak L}$ and this gives a generalized complex structure on $\g$.

\medskip

Define ${\mathfrak U}^{\bullet}\subset  \wedge ^{\bullet}\g^{\ast}_{\C}$ such that  ${\mathfrak U}^{n}=\langle \rho\rangle$ and
${\mathfrak U}^{n-r}=\wedge^{r} \overline{\mathfrak L} \cdot {\mathfrak U}^{n}$.
Then, by the integrability, we have $d{\mathfrak U}^{j}\subset {\mathfrak U}^{j-1}\oplus {\mathfrak U}^{j+1}$. 
We consider the decomposition  $d=\del+\delbar$ such that 
$\del\colon {\mathfrak U}^{j}\to  {\mathfrak U}^{j+1}$ and $\delbar\colon {\mathfrak U}^{j}\to  {\mathfrak U}^{j-1}$.
Hence we have the bi-differential $\Z$-graded complex $({\mathfrak U}^{\bullet},\del,\delbar)$.
We define
\begin{eqnarray*}
 GH^{\bullet}_{\del}(\g) &:=& \frac{\ker\left(\del\colon \mathfrak U^\bullet \to \mathfrak U^{\bullet+1}\right)}{\imm\left(\del\colon \mathfrak  U^{\bullet-1} \to \mathfrak U^{\bullet}\right)}, \\
 GH^{\bullet}_{\delbar}(\g) &:=& \frac{\ker\left(\delbar\colon \mathfrak U^\bullet \to \mathfrak U^{\bullet-1}\right)}{\imm\left(\delbar\colon \mathfrak  U^{\bullet+1} \to \mathfrak U^{\bullet}\right)} \;,
\\
GH^{\bullet}_{BC}(\g) &:=& \frac{\ker\left(\del\colon\mathcal  U^{\bullet} \to \mathfrak  U^{\bullet+1} \right) \cap \ker\left(\delbar\colon\mathfrak U^{\bullet} \to\mathfrak U^{\bullet-1} \right)}{\imm \left(\del\delbar \colon\mathfrak U^{\bullet} \to\mathfrak U^{\bullet} \right)} \;, \\[5pt]
GH^{\bullet}_{A}(\g) &:=& \frac{\ker \left( \del\delbar \colon \mathfrak U^{\bullet} \to \mathfrak U^{\bullet} \right) }{\imm \left(\del \colon\mathfrak U^{\bullet-1} \to \mathfrak U^{\bullet} \right) + \imm \left( \delbar \colon\mathfrak U^{\bullet+1} \to\mathfrak U^{\bullet} \right) } \;.
\end{eqnarray*}

By the integrability $d\rho=(X+\zeta)\rho$ and from the identification of ${\mathfrak L}^{\ast}= \overline{\mathfrak L}$ by the pairing, we can consider $\langle \rho \rangle$ as a ${\mathfrak L}$-module and
we can identify $({\mathfrak U}^{n-\bullet},\delbar)$ with $\wedge ^{\bullet}{\mathfrak L}^{\ast}\otimes \langle \rho \rangle$ as a cochain complex of 
the Lie algebra ${\mathfrak L}$ with values in the module $\langle \rho \rangle$ (cf. \cite[page 98]{gualtieri-annals}).
In particular, if $d\rho=0$, then we have ${\mathfrak U}^{n-\bullet}\cong\wedge ^{\bullet}{\mathfrak L}^{\ast}$.

\medskip

We consider the following special case for using techniques of  spectral sequences.

\begin{ex}\label{ideal}
Let $\g$ be a Lie algebra and ${\mathfrak h}\subset \g$ an ideal of $\g$.
Consider the differential graded algebra extension 
\[\wedge^{\bullet} \g^{\ast}=\wedge^{\bullet}\left(\g/{\mathfrak h}\right)^{\ast}\otimes \wedge^{\bullet}{\mathfrak h}^{\ast}
\]
dualizing the Lie algebra extension
\[\xymatrix{
0\ar[r]& {\mathfrak h}\ar[r]&\g\ar[r]&\g/{\mathfrak h}\ar[r]&0.
 }
\]
We assume that:
\begin{itemize}
\item $\g/{\mathfrak h}$ admits a complex structure $J$;
\item we have a closed $2$-form $\omega\in \wedge^{2} \g^{\ast}$
yielding $\omega\in \wedge^{2}{\mathfrak h}^{\ast}$ non-degenerate form on $\mathfrak h$.
\end{itemize}
Consider the $\pm i$-eigenspace decomposition
\[\left(\g/{\mathfrak h}\right)\otimes \C=\left(\g/{\mathfrak h}\right)^{1,0}\oplus\left(\g/{\mathfrak h}\right)^{0,1}.
\]
Take a basis $Z_{1},\dots,Z_{k}$ of $\left(\g/{\mathfrak h}\right)^{1,0}$ and the dual basis $\theta_{1},\dots,\theta_{k}$ of $\left(\g/{\mathfrak h}\right)^{\ast 1,0}$.
Then we have the  non-degenerate integrable pure form 
\[\rho=e^{i\omega}\theta_{1}\wedge \dots\wedge \theta_{k}.
\]
We have
\[{\mathfrak L}=\langle \theta_{1},\dots,\theta_{k}, \overline{Z_{1}},\dots,\overline{Z_{k}}\rangle\oplus \{X-i\omega(X): X\in {\mathfrak h}\otimes \C\}
\]
and consider the subspace ${\mathfrak S}= \{X-i\omega(X) \;\vert\; X\in {\mathfrak h}\otimes \C\}$.
By $d\omega=0$ in $\wedge^{\bullet} \g^{\ast}$ and $\omega\in \wedge^{2}{\mathfrak h}^{\ast}$,
${\mathfrak S}$ is an ideal of ${\mathfrak L}$.
We have ${\mathfrak L}/ {\mathfrak S}\cong {\mathfrak L}_{J}= \left(\g/\mathfrak h\right)^{0,1}\oplus \left(\g/\mathfrak h\right)^{\ast 1,0}$.
We have the isomorphism ${\mathfrak h}\otimes \C\ni X\mapsto X-i\omega(X)\in {\mathfrak S}$.

By the Hochschild-Serre spectral sequence,  we have the spectral sequence $\left\{ E^{\bullet,\bullet}_{r} \right\}_{r}$ which converges to $H^{\bullet}({\mathfrak L})$ such that 
\[E_{2}^{p,q}=H^{p}({\mathfrak L}/{\mathfrak S}, H^{q}({\mathfrak S})).
\]

\end{ex}

\section{de Rham and Dolbeault Cohomology of nilmanifolds}

Let $G$ be a connected simply-connected nilpotent Lie group and $\g$ the Lie algebra of $G$.
A $\Q$-structure of $\g$ is a $\Q$-subalgebra $\g_{\Q}\subset \g$ such that $\g_{\Q}\otimes \R=\g$.
It is known that $\g$ admits a $\Q$-structure if and only if $G$ admits a lattice (namely, a cocompact discrete subgroup), see, e.g., \cite{Raghunathan}.
More precisely, considering the exponential map $\exp\colon \g\to G$ which is an diffeomorphism, we can say that: 
\begin{itemize}
\item for a $\Q$-structure $\g_{\Q}\subset \g$, taking a basis $X_{1},\dots, X_{n}$ of $\g_{\Q}$, the group  generated by $\exp (\Z\langle X_{1},\dots, X_{n}\rangle)$ is a lattice in $G$;
\item for a lattice $\Gamma\subset G$, the $\Z$-span of $\exp^{-1}(\Gamma)$ is a $\Q$-structure of $\g$.
\end{itemize}
If $G$ admits a lattice $\Gamma$, we call $\Gamma\backslash G$ a nilmanifold.

\medskip

We suppose that $G$ admits a lattice $\Gamma$ and consider the $\Q$-structure $\g_{\Q}\subset \g$ induced by $\Gamma$ as above.
Let $\mathfrak h\subset \g$ be a subalgebra and $H=\exp(\mathfrak h)$.
We suppose that $\g_{\Q}\cap \mathfrak h$ is a $\Q$-structure of $\mathfrak h$.
Then $H\cap \Gamma$ is a lattice of $H$, see \cite[Remark 2.16]{Raghunathan}.
If $\mathfrak h$ is an ideal, then $H$ is normal and we obtain
the fibre bundle $\Gamma\backslash G\to \Gamma H\backslash G$ with the fibre $\Gamma\cap H\backslash H$.

\medskip

For a nilmanifold $\Gamma\backslash G$, 
regarding the cochain complex $\wedge^{\bullet}\g^\ast$ as the space of left-invariant differential forms on $\Gamma\backslash G$,
we have the inclusion
\[\iota :\wedge^{\bullet}\g^\ast\to \wedge^{\bullet}\Gamma\backslash G.
\]
\begin{thm}[\cite{nomizu}]\label{nomth}
The inclusion $\iota \colon\wedge^{\bullet}\g^\ast\to \wedge^{\bullet}\Gamma\backslash G$ induces a cohomology isomorphism
\[H^{\bullet}(\g)\cong H^{\bullet}(\Gamma\backslash G).
\]
\end{thm}
Suppose that $\g$ admits a complex structure $J$.
Then we can define the Dolbeault complex $\wedge^{\bullet,\bullet}\g^{\ast}\otimes\C$ of $(\g,J)$.
Consider the left-invariant complex structure on the nilmanifold $\Gamma\backslash G$ induced by $J$
and the Dolbeault complex $\wedge^{\bullet,\bullet}\Gamma\backslash G$.
Then we have the inclusion $\iota\colon \wedge^{\bullet,\bullet}\g^{\ast}\otimes\C\to \wedge^{\bullet,\bullet}\Gamma\backslash G$.

Let ${\mathfrak L}_{J}=\g^{0,1}\oplus \g^{\ast1,0}$.
We consider the Lie algebroid $L_{\Gamma\backslash G}=T^{0,1}\Gamma\backslash G\oplus T^{\ast1,0}\Gamma\backslash G$ for the generalized complex structure associated with the complex structure on $\Gamma\backslash G$.
Then we have $\mathcal{C}^{\infty}(\wedge^{\bullet} L_{\Gamma\backslash G}^{\ast})=\mathcal{C}^{\infty}(\Gamma\backslash G)\otimes \wedge^{\bullet}{\mathfrak L}_{J}^{\ast}$ and we have the inclusion 
\[\kappa\colon \wedge^{\bullet}{\mathfrak L}_{J}^{\ast}\to \mathcal{C}^{\infty}(\wedge^{\bullet} L_{\Gamma\backslash G}^{\ast}) .
\]

\begin{prop}[\cite{kasuya-dg}]\label{tri}
If the inclusion $\iota\colon \wedge^{\bullet,\bullet}\g^{\ast}\otimes\C\to \wedge^{\bullet,\bullet}\Gamma\backslash G$ induces an isomorphism on the Dolbeault cohomology,
then the inclusion 
$\kappa\colon \wedge^{\bullet}{\mathfrak L}_{J}^{\ast}\to \mathcal{C}^{\infty}(\wedge^{\bullet} L_{\Gamma\backslash G}^{\ast})$
induces a cohomology isomorphism.
\end{prop}

\medskip

Let $W$ be a complex valued $\g$-module.
We regard $W$ as a $\g^{0,1}$-module and so a ${\mathfrak L}_{J}$-module.
We consider the cochain complex $ \wedge^{\bullet}{\mathfrak L}_{J}^{\ast}\otimes W $ of the Lie algebra with values in the module $W$.
Consider the flat complex vector bundle $\bf W$ over $\Gamma\backslash G$ given by $W$.
We regard $\bf W$ as a holomorphic bundle over {\color{blue}{$\Gamma  \backslash G$}} and so a generalized holomorphic bundle on $\Gamma\backslash G$.
We have  $\mathcal{C}^{\infty}(\wedge^{\bullet} L_{\Gamma\backslash G}^{\ast}\otimes \mathbf{W})=\mathcal{C}^{\infty}(\Gamma\backslash G)\otimes \wedge^{\bullet}{\mathfrak L}_{J}^{\ast}\otimes W$
and we have the inclusion
\[\kappa:\wedge^{\bullet}{\mathfrak L}_{J}^{\ast}\otimes W\to \mathcal{C}^{\infty}(\wedge^{\bullet} L_{\Gamma\backslash G}^{\ast}\otimes \bf W).
\]
\begin{prop}\label{enge}
We suppose that the inclusion $\iota: \wedge^{\bullet,\bullet}\g^{\ast}\otimes\C\to \wedge^{\bullet,\bullet}\Gamma\backslash G$ induces an isomorphism on the Dolbeault cohomology
and $W$ is a nilpotent $\g$-module.
Then the inclusion
\[\kappa:\wedge^{\bullet}{\mathfrak L}_{J}^{\ast}\otimes W\to \mathcal{C}^{\infty}(\wedge^{\bullet} L_{\Gamma\backslash G}^{\ast}\otimes \bf W)
\]
induces a cohomology isomorphism.
\end{prop}

\begin{proof}
The proof is by induction on the dimension of $W$.

Suppose first $\dim W=1$.
Then $W$ is the trivial $\g$-module and hence the statement follows from Proposition \ref{tri}.

In case $\dim W=n>1$, by Engel's theorem, we have a $(n-1)$-dimensional $\g$-submodule $\tilde W \subset W$ such that the quotient $W/\tilde W$ is the trivial submodule.
The exact sequence
\[\xymatrix{
0\ar[r]&\tilde W\ar[r]&W\ar[r]&W/\tilde{W}\ar[r]&0
}
\]
gives the commutative diagram
\[\xymatrix{
0\ar[r]& \wedge^{\bullet}{\mathfrak L}_{J}^{\ast}\otimes \tilde W\ar[r]\ar[d]&\wedge^{\bullet}{\mathfrak L}_{J}^{\ast}\otimes W\ar[r]\ar[d] &\wedge^{\bullet}{\mathfrak L}_{J}^{\ast}\otimes W/\tilde W\ar[r]\ar[d]&0\\
0\ar[r]& \mathcal{C}^{\infty}(\wedge^{\bullet} L_{\Gamma\backslash G}^{\ast}\otimes {\bf \tilde W})\ar[r]& \mathcal{C}^{\infty}(\wedge^{\bullet} L_{\Gamma\backslash G}^{\ast}\otimes {\bf W})\ar[r] & \mathcal{C}^{\infty}(\wedge^{\bullet} L_{\Gamma\backslash G}^{\ast}\otimes {\bf W}/{\bf \tilde W})\ar[r]&0
}
\]
such that the horizontal sequences are exact.
Considering the long exact sequence of cohomologies, by the five lemma, the proposition follows inductively.
\end{proof}

\section{Left-invariant generalized complex structures on nilmanifolds}
Let $G$ be a connected simply-connected nilpotent Lie group and $\g$ the Lie algebra of $G$.
We assume that $G$ admits a lattice $\Gamma$.
We consider the nilmanifold $\Gamma\backslash G$.

We assume that $\g$ admits a generalized complex structure associated with a non-degenerate integrable pure form $\rho\in \wedge ^{\bullet}\g^{\ast}_{\C}$ of type $k$.
Then we have the left-invariant generalized complex structure $\mathcal J$ 
of type $k$ on the nilmanifold $\Gamma\backslash G$.

Consider the  the bi-differential $\Z$-graded complexes $({\mathfrak U}^{\bullet},\del,\delbar)$ associated with $(\g,\rho)$ and $({\mathcal U}^{\bullet},\del,\delbar)$ associated with $(\Gamma\backslash G,\mathcal J)$.
Then  the inclusion $\iota:\wedge ^{\bullet}\g^{\ast}_{\C}\to \wedge^{\bullet} \Gamma\backslash G\otimes \C$ can be considered as a homomorphism $({\mathfrak U}^{\bullet},\del,\delbar)\to ({\mathcal U}^{\bullet},\del,\delbar) $ of bi-differential $\Z$-graded complexes.

\begin{prop}
There exists a homomorphism $\mu\colon ({\mathcal U}^{\bullet},\del,\delbar)\to  ({\mathfrak U}^{\bullet},\del,\delbar)$ such that $\mu \circ \iota={\rm id}$.
Hence the induced map $\iota\colon GH_{\delbar}(\g)\to GH_{\delbar}(\Gamma\backslash G)$ is injective.
\end{prop}

\begin{proof}
Let $d\nu$ be a bi-invariant volume form such that $\int_{\Gamma\backslash G}d\nu =1$.
We define the map $\mu\colon \wedge^{\bullet}\Gamma\backslash G\to  \wedge^{\bullet} \g^{\ast}_{\C}$ as follows: for $\alpha\in \wedge^{\bullet}\Gamma\backslash G$, the left-invariant form $\mu(\alpha)$ is defined by
\[\mu(\alpha)(X_{1},\dots ,X_{p})=\int_{\Gamma\backslash G}\alpha(\tilde X_{1},\dots ,\tilde X_{p})d\nu \;,
\]
where $\tilde X_{1},\dots ,\tilde X_{p}$ are vector fields on $\Gamma\backslash G$ induced by $X_{1},\dots X_{p}\in\g$.
Then we have $d\circ \mu =\mu\circ d$ and $\mu \circ \iota ={\rm id}$.
We have $\mu ({\mathcal U}^{\bullet})\subset {\mathfrak U}^{\bullet} $.
We consider $\mathcal J$ as an operator on ${\mathcal U}^{\bullet}$ such that
\[\mathcal J(\alpha)=ip\alpha
\]
for $\alpha\in \mathcal{U}^{p}$.
Then we have $d{\mathcal J}-{\mathcal J}d=-i(\del-\delbar)$, see \cite{cavalcanti-phd, cavalcanti-jgp}.
By $\mu\circ{\mathcal J}={\mathcal J}\circ\mu$, we have $\mu  \circ \del =\del\circ \mu$ and $\mu  \circ \delbar =\delbar\circ \mu$.
Hence the proposition follows.
\end{proof}

\begin{cor}
If the induced map $\iota\colon GH_{\delbar}(\g)\to GH_{\delbar}(\Gamma\backslash G)$ is an isomorphism,
then the induced maps
$\iota\colon GH_{\del}(\g)\to GH_{\del}(\Gamma\backslash G)$ and 
$\iota\colon GH_{BC}(\g)\to GH_{BC}(\Gamma\backslash G)$  are also isomorphisms.
\end{cor}

\begin{proof}
By using the complex conjugation, we can easily prove that $\iota\colon GH_{\del}(\g)\to GH_{\del}(\Gamma\backslash G)$ is an isomorphism if $\iota: GH_{\delbar}(\g)\to GH_{\delbar}(\Gamma\backslash G)$ is an isomorphism.

 Now, \cite[Corollary 1.2]{angella-kasuya-SG} implies that if
 $\iota\colon GH_{\delbar}(\g)\to GH_{\delbar}(\Gamma\backslash G)$ and $\iota\colon GH_{\del}(\g)\to GH_{\del}(\Gamma\backslash G)$ are isomorphisms, then  
$\iota\colon GH_{BC}(\g)\to GH_{BC}(\Gamma\backslash G)$ is an isomorphism.
\end{proof}

By Theorem \ref{nilcal}, in our settings, we have isomorphisms
$GH^{n-\bullet}_{\delbar}(\g)\cong H^{\bullet}(\mathfrak L)$ and
$GH^{n-\bullet}_{\delbar}(\Gamma\backslash G)\cong H^{\bullet}( L)$.
Thus,  $H^{\bullet}(\mathfrak L) \cong H^{\bullet}( L)$ if and only if $GH^{n-\bullet}_{\delbar}(\g)\cong GH^{n-\bullet}_{\delbar}(\Gamma\backslash G)$.

\medskip

Let $G$ be a connected simply-connected nilpotent Lie group and $\g$ the Lie algebra of $G$.
We suppose that $G$ admits a lattice $\Gamma$ and consider the $\Q$-structure $\g_{\Q}\subset \g$ induced by $\Gamma$.
We assume  that there exists an ideal  $\mathfrak h\subset \g$ so that:
\begin{enumerate}
\item $\g_{\Q}\cap \mathfrak h$ is a $\Q$-structure of $\mathfrak h$;
\item $\g/{\mathfrak h}$ admits a complex structure $J$;
\item we have a closed $2$-form $\omega\in \wedge^{2} \g^{\ast}$
yielding $\omega\in \wedge^{2}{\mathfrak h}^{\ast}$ non-degenerate form on $\mathfrak h$.
\end{enumerate}
Then, as in Example \ref{ideal}, we obtain the non-degenerate integrable pure form $\rho\in \wedge ^{\bullet}\g^{\ast}_{\C}$
and the Lie algebra ${\mathfrak L}$ and its ideal $\mathfrak S$.
We  obtain the symplectic fibre bundle $\Gamma\backslash G\to \Gamma H\backslash G$ over the complex base $\Gamma H\backslash G$ with the symplectic fibre $\Gamma\cap H\backslash H$ as in Section \ref{fisp}.
The left-invariant generalized complex structure given by $\rho$ is the generalized complex structure constructed in  Section \ref{fisp}.
Consider the Lie algebroids $L$ and $S$ as in   Section \ref{fisp}.
Then ${\mathfrak L}$ and ${\mathfrak S}$ give the global frame of $L$ and $S$ respectively.

Consider the cochain complex $\wedge^{\bullet}{\mathcal L}^{\ast}$ and $\mathcal{C}^{\infty}(\wedge^{\bullet} L^{\ast})$.
Then we have $\mathcal{C}^{\infty}(\wedge^{\bullet} L^{\ast})=\mathcal{C}^{\infty}(\Gamma\backslash G)\otimes \wedge^{\bullet}{\mathcal L}^{\ast}$ and we have the inclusion
\[\wedge^{\bullet}{\mathcal L}^{\ast}\to \mathcal{C}^{\infty}(\wedge^{\bullet} L^{\ast}).
\]
For the ideal ${\mathfrak S}$, we consider the filtration
\[F^{p}\wedge^{p+q}{\mathfrak L}^{\ast}=
\left\{\phi\in \wedge^{p+q}{\mathfrak L}^{\ast} \;\middle\vert\; \omega(X_{1},\dots, X_{p+q})=0 \text{ for } X_{\ell_{1}},\dots, X_{\ell_{q+1}}\in {\mathfrak S}\right\}.
\]
This filtration gives the spectral sequence $\left\{ \,^{\prime}E_{r}^{\bullet,\bullet} \right\}_{r}$ which converges to $H^{\bullet}({\mathfrak L})$ such that 
\[\,^{\prime}E_{2}^{p,q}=H^{p}({\mathfrak L}/{\mathfrak S},H^{q}({\mathfrak S})).
\]
By the identifications ${\mathfrak L}/ {\mathfrak S}\cong {\mathfrak L}_{J}$ and ${\mathfrak S}=\mathfrak h\otimes \C$,
we have
\[\,^{\prime}E_{2}^{p,q}=H^{p}({\mathfrak L}_{J},H^{q}(\mathfrak h\otimes \C)).
\]
The filtration $F^{p}\wedge^{\bullet}{\mathfrak L}^{\ast}$ can be extended to the filtration of $ \mathcal{C}^{\infty}(\wedge^{\bullet} L^{\ast})$ constructed in  Section \ref{fisp}.
Hence the inclusion $\wedge^{\bullet}{\mathcal L}^{\ast}\to \mathcal{C}^{\infty}(\wedge^{\bullet} L^{\ast})$
induces the spectral sequence homomorphism $\,^{\prime}E_{\bullet}^{\bullet,\bullet}\to E_{\bullet}^{\bullet,\bullet}$ such that the homomorphism $\,^{\prime}E_{2}^{\bullet,\bullet}\to E_{2}^{\bullet,\bullet}$ is identified with the map
\[H^{p}({\mathfrak L}_{J},H^{q}(\mathfrak h\otimes \C))\to H^{p}(L_{ \Gamma H\backslash G}, {\bf H}^{q}(\Gamma\cap H\backslash H)).
\]
By  Theorem \ref{nomth},
the flat bundle ${\bf H}^{q}(\Gamma\cap H\backslash H)$
over $\Gamma H\backslash G$ is derived from the $\g/\mathfrak h$-module $H^{q}(\mathfrak h\otimes \C)$.
The $\g/\mathfrak h$-module $H^{q}(\mathfrak h\otimes \C)$ being induced by the adjoint representation on the nilpotent Lie algebra $\g$,
it is a nilpotent $\g/\mathfrak h$-module.
If $\iota: \wedge^{\bullet,\bullet}(\g/\mathfrak h)^{\ast}\otimes\C\to \wedge^{\bullet,\bullet}\Gamma H\backslash G$ induces an isomorphism on the Dolbeault cohomology,
then  the homomorphism $\,^{\prime}E_{2}^{\bullet,\bullet}\to E_{2}^{\bullet,\bullet}$ is an isomorphism.

Hence, by Proposition \ref{enge}, we obtain the following result.

\begin{thm}\label{lieiso}
Let $G$ be a connected simply-connected nilpotent Lie group and $\g$ the Lie algebra of $G$.
We suppose that $G$ admits a lattice $\Gamma$ and consider the $\Q$-structure $\g_{\Q}\subset \g$ induced by $\Gamma$.
We assume  that there exists an ideal  $\mathfrak h\subset \g$ so that:
\begin{enumerate}
\item $\g_{\Q}\cap \mathfrak h$ is a $\Q$-structure of $\mathfrak h$;
\item $\g/{\mathfrak h}$ admits a complex structure $J$;
\item we have a closed $2$-form $\omega\in \wedge^{2} \g^{\ast}$
yielding $\omega\in \wedge^{2}{\mathfrak h}^{\ast}$ non-degenerate form on $\mathfrak h$;
\item $\iota\colon \wedge^{\bullet,\bullet}(\g/\mathfrak h)^{\ast}\otimes\C\to \wedge^{\bullet,\bullet}\Gamma H\backslash G$ induces an isomorphism on the Dolbeault cohomology (e.g., $J$ is bi-invariant, Abelian, or rational, i.e., $J(\g_{\Q}/{\mathfrak h}\cap \g_{\Q})\subset \g_{\Q}/{\mathfrak h}\cap \g_{\Q}$).
\end{enumerate}
Then the inclusion \[\wedge^{\bullet}{\mathcal L}^{\ast}\to \mathcal{C}^{\infty}(\wedge^{\bullet} L^{\ast})
\]
induces a cohomology isomorphism.
\end{thm}

\begin{cor}\label{cor:lieiso}
In the same assumptions of Theorem \ref{lieiso},
the inclusion $\iota\colon({\mathfrak U}^{\bullet},\del,\delbar)\to ({\mathcal U}^{\bullet},\del,\delbar) $ induces isomorphisms
$GH_{\delbar}(\g)\cong GH_{\delbar}(\Gamma\backslash G)$, and
$GH_{\del}(\g)\cong GH_{\del}(\Gamma\backslash G)$, and 
$GH_{BC}(\g)\cong GH_{BC}(\Gamma\backslash G)$.
\end{cor}

\section{Deformation and cohomology}
We consider a nilmanifold $\Gamma\backslash G$ with a left-invariant generalized complex structure $\mathcal J$.
We consider the Lie algebra $\mathfrak L\subset (\g\oplus \g^{\ast})\otimes \C$ and the cochain complex $\wedge^{\bullet}\mathfrak L^{\ast}$.
By the identification $ \overline{\mathfrak L}=\mathfrak L^{\ast}$,
we have the bracket on  $\mathfrak L^{\ast}$.
Consider the Schouten bracket on $\wedge^{\bullet}\mathfrak L^{\ast}$.
Then, for the inclusion $\wedge^{\bullet}{\mathcal L}^{\ast}\subset \mathcal{C}^{\infty}(\wedge^{\bullet} L^{\ast})$, the Schouten bracket on $\wedge^{\bullet}\mathfrak L^{\ast}$ can be extended to the Schouten bracket on $\mathcal{C}^{\infty}(\wedge^{\bullet} L^{\ast})$.

We assume that we have a smooth family $\epsilon(t)\in \wedge^{2} L^{\ast}$ which satisfy
the Maurer-Cartan equation
\[d_{\mathfrak L}\epsilon+\frac{1}{2}[\epsilon,\epsilon]=0
\]
such that $\epsilon(0)=0$.
Then we have deformations $\mathcal J_{\epsilon(t)}$ of $\mathcal J$.

\begin{thm}\label{thm:coh-def}
Let $\Gamma\backslash G$ be a nilmanifold with a left-invariant generalized complex structure $\mathcal J$; denote by $\g$ be the Lie algebra of $G$.
If the isomorphism $GH_{\delbar}(\g)\cong GH_{\delbar}(\Gamma\backslash G)$ holds on the original generalized complex structure $\mathcal J$,
then the same isomorphism holds on the deformed generalized complex structure $\mathcal J_{\epsilon(t)}$
for sufficiently small $t$.
\end{thm}

\begin{proof}
Take a smooth family of generalized Hermitian metrics for the
generalized complex structures $\mathcal J_{\epsilon(t)}$.
We obtain the smooth family $\Delta_{\bar\partial}(t)$ of elliptic operators on $\wedge^{\bullet} \Gamma\backslash G\otimes\C $ such that $\Delta_{\bar\partial}(t)(\wedge^{\bullet}\g^{\ast}\otimes \C)\subset \wedge^{\bullet}\g^{\ast}\otimes \C$.

Take a Hermitian metric on $\g\otimes \C$ and extend it to $T\Gamma\backslash G\otimes\C$.
Consider the completion $W^{0}(\wedge^{\bullet} \Gamma\backslash G\otimes\C)$ with respect to the $L^{2}$-norm.
Consider the orthogonal complement $(\ker \Delta_{\bar\partial}(t))^{\perp}$ in $W^{0}(\wedge^{\bullet} \Gamma\backslash G\otimes\C)$.
It is known that for  sufficiently small $t$,
we have $(\ker \Delta_{\bar\partial}(0))^{\perp}\cap \ker \Delta_{\bar\partial}(t)=0$.

We can easily show that any cohomology class in $GH_{\delbar}(\g)$ admits a  unique representative in $\ker \Delta_{\bar\partial}(0)$.
Hence, by the isomorphism $GH_{\delbar}(\g)\cong GH_{\delbar}(\Gamma\backslash G)$, we have $\ker \Delta_{\bar\partial}(0)\subset \wedge^{\bullet}\g^{\ast}\otimes \C$.
This implies that $(\wedge^{\bullet}\g^{\ast}\otimes \C)^{\perp}\subset (\ker \Delta_{\bar\partial}(0))^{\perp}$.
By $(\ker \Delta_{\bar\partial}(0))^{\perp}\cap \ker \Delta_{\bar\partial}(t)=0$, we have $\ker \Delta_{\bar\partial}(t)\subset \wedge^{\bullet}\g^{\ast}\otimes \C$.
Hence, on the deformed generalized complex structure  $\mathcal J_{\epsilon(t)}$, 
 any cohomology class in $GH_{\delbar}(\Gamma\backslash G)$ admits a   representative in $\wedge^{\bullet}\g^{\ast}\otimes \C$ and the theorem follows.
\end{proof}

\begin{thm}\label{thm:def-inv}
Let $\Gamma\backslash G$ be a nilmanifold with a left-invariant generalized complex structure $\mathcal J$; denote by $\g$ be the Lie algebra of $G$.
If the isomorphism $GH_{\delbar}(\g)\cong GH_{\delbar}(\Gamma\backslash G)$ holds on the original generalized complex structure $\mathcal J$,
then any sufficiently small deformation of generalized complex structure is equivalent to a left-invariant complex structure $\mathcal J_{\epsilon}$ with $\epsilon \in \wedge^{2} \mathfrak L^{\ast}$ satisfying the Maurer-Cartan equation.
\end{thm}

\begin{proof}
By the isomorphism $GH_{\delbar}(\g)\cong GH_{\delbar}(\Gamma\backslash G)$, we have the isomorphism
$H^{\ast}(\mathfrak L) \cong H^{\ast}( L)$.

Take a Hermitian metric on $\mathfrak L$.
Since $\mathfrak L$ gives the global frame of $L$, it gives a Hermitian metric on $L$.
Consider the adjoint operator $d_{L}^{\ast}$, the Laplacian operator $\Delta_{L}=d_{L}d_{L}^{\ast}+d_{L}^{\ast}d_{L}$, the projection $H: \mathcal{C}^{\infty}(\wedge^{\bullet} L^{\ast})\to \ker \Delta_{L}$ and the Green operator $G$.
Obviously, these operators can be extended to $\wedge^{\bullet} \mathfrak L^{\ast}$.
Since $\wedge^{\bullet} \mathfrak L^{\ast}$ is finite dimensional, we can easily prove that any cohomology class in $H^{\ast}(\mathfrak L)$ admits a unique representative in $\ker \Delta_{L}$.
Hence, combining with the Hodge theory on the elliptic complex $ (\mathcal{C}^{\infty}(\wedge^{\bullet} L^{\ast}),d_{L})$,
we have $\ker \Delta_{L}\subset \wedge^{\bullet} \mathfrak L^{\ast}$.
Hence,
for $\epsilon_{1}\in \ker \Delta_{L} $, the 
formal power series $\epsilon(\epsilon_{1})$ as in Theorem \ref{defku} is valued in $\wedge^{\bullet} \mathfrak L^{\ast}$.
Thus the theorem follows from Theorem \ref{defku}.
\end{proof}

\section{Example: the Kodaira-Thurston manifold}\label{sec:kt}

We consider the  real  Heisenberg group $H_{3}(\R)$ which is the group of matrices of the form
\[\left(
\begin{array}{ccc}
1& x&z  \\
0&    1&y\\
0&0&1  
\end{array}
\right)
\]
where  $x,y, z\in \R$.
Then $H_{3}(\R)$ admits the lattice $H_{3}(\Z)=GL_{3}(\Z)\cap H_{3}(\R) $.
We consider the Lie group $H_{3}(\R)\times \R$ with the lattice $H_{3}(\Z)\times \Z$.

Let $\g=\langle X_{1}, X_{2},X_{3},X_{4}\rangle$ such that $[X_{1},X_{2}]=X_{3}$ and other brackets are $0$.
Then $\g$ is the Lie algebra of $H_{3}(\R)\times \R$ and the basis $X_{1}, X_{2},X_{3},X_{4}$ gives the $\Q$-structure associated with the lattice $H_{3}(\Z)\times \Z$.
Consider the ideal $\mathfrak h=\left\langle X_{2},X_{3} \right\rangle$.
In this case, the assumptions in Theorem \ref{lieiso} hold.

Take the dual basis $\{x_{1},x_{2},x_{3},x_{4}\}$ of $\{X_{1}, X_{2},X_{3},X_{4}\}$ and consider 
$\wedge^{\bullet}\g_{\C}^{\ast}=\wedge^{\bullet}\langle x_{1},x_{2},x_{3},x_{4}\rangle$. 
Consider the non-degenerate integrable pure form 
\[
\rho=e^{ix_{2}\wedge x_{3}}\wedge (x_{1}+ix_{4})
\]
of type $1$.
We have 
\[{\mathfrak L}=\langle X_{1}+iX_{4}, x_{1}+ix_{4}, X_{2}-ix_{3}, X_{3}+ix_{2}\rangle.
\]
In this case, 
 we have 
${\mathfrak S}=\langle X_{2}-ix_{3}, X_{3}+ix_{2}\rangle$ and ${\mathfrak S}$ is  an ideal.
We obtain
\begin{eqnarray*}
{\mathfrak U}^{2}&=&\langle \rho\rangle\\[5pt]
{\mathfrak U}^{1}&=&\langle e^{ix_{2}\wedge x_{3}},\; e^{ix_{2}\wedge x_{3}}\wedge (x_{1}+ix_{4})\wedge (x_{1}-ix_{4}),\; (x_{1}+ix_{4})\wedge x_{3}, \\[5pt]
&& (x_{1}+ix_{4})\wedge x_{2}\rangle\\[5pt]
{\mathfrak U}^{0}&=&\langle e^{ix_{2}\wedge x_{3}}\wedge (x_{1}-ix_{4}),\; x_{3},\; x_{2},\; x_{3}\wedge (x_{1}+ix_{4})\wedge (x_{1}-ix_{4}), \\[5pt]
&&x_{2}\wedge (x_{1}+ix_{4})\wedge (x_{1}-ix_{4}),\; e^{-ix_{2}\wedge x_{3}}\wedge (x_{1}+ix_{4})
\rangle\\[5pt]
{\mathfrak U}^{-1}&=&\langle e^{-ix_{2}\wedge x_{3}},\; e^{-ix_{2}\wedge x_{3}}\wedge (x_{1}+ix_{4})\wedge (x_{1}-ix_{4}),\;  (x_{1}-ix_{4})\wedge x_{3},\\[5pt]
&& (x_{1}-ix_{4})\wedge x_{2}\rangle\\[5pt]
{\mathfrak U}^{-2}&=&\langle \bar\rho\rangle.
\end{eqnarray*}

We have that the only non-trivial differentials are
\begin{eqnarray*}
d \left((x_{1}+ix_{4})\wedge x_{3}\right) = \delbar \left( (x_{1}+ix_{4})\wedge x_{3} \right) = i x_1 \wedge x_2 \wedge x_4 \;, \\[5pt]
d \left(x_{3}\right) = -\frac{1}{2} (x_1+i x_4)\wedge x_2 - \frac{1}{2} (x_1-i x_4)\wedge x_2 \;, \\[5pt]
d \left((x_{1}-ix_{4})\wedge x_{3}\right) = \del \left( (x_{1}+ix_{4})\wedge x_{3} \right) = i x_1 \wedge x_2 \wedge x_4 \;.
\end{eqnarray*}

Define the Kodaira-Thurston manifold as the compact quotient
$$ M \;:=\; \left. \left( H_{3}(\Z)\times \Z \right) \middle\backslash \left( H_{3}(\R)\times \R \right) \right. \;.$$
By Corollary \ref{cor:lieiso}, we get:
\begin{eqnarray*}
 GH_{\delbar}^{2}(M) &=& \langle [\rho] \rangle\\[5pt]
 GH_{\delbar}^{1}(M) &=& \langle [e^{ix_{2}\wedge x_{3}}],\; [e^{ix_{2}\wedge x_{3}}\wedge (x_{1}+ix_{4})\wedge (x_{1}-ix_{4})] \rangle\\[5pt]
 GH_{\delbar}^{0}(M) &=& \langle [e^{ix_{2}\wedge x_{3}}\wedge (x_{1}-ix_{4})],\; [x_{2}],\; [x_{3}\wedge (x_{1}+ix_{4})\wedge (x_{1}-ix_{4})], \\[5pt]
  && [e^{-ix_{2}\wedge x_{3}}\wedge (x_{1}+ix_{4})] \rangle\\[5pt]
 GH_{\delbar}^{-1}(M) &=& \langle [e^{-ix_{2}\wedge x_{3}}],\; [e^{-ix_{2}\wedge x_{3}}\wedge (x_{1}+ix_{4})\wedge (x_{1}-ix_{4})] \rangle\\[5pt]
 GH_{\delbar}^{-2}(M) &=& \langle [\bar\rho] \rangle \;,
\end{eqnarray*}
and
\begin{eqnarray*}
 GH_{BC}^{2}(M) &=& \langle [\rho] \rangle\\[5pt]
 GH_{BC}^{1}(M) &=& \langle [e^{ix_{2}\wedge x_{3}}],\; [e^{ix_{2}\wedge x_{3}}\wedge (x_{1}+ix_{4})\wedge (x_{1}-ix_{4})],\; \\[5pt]
  && [(x_{1}+ix_{4})\wedge x_{2}] \rangle\\[5pt]
 GH_{BC}^{0}(M) &=& \langle [e^{ix_{2}\wedge x_{3}}\wedge (x_{1}-ix_{4})],\; [x_{2}],\; [x_{3}\wedge (x_{1}+ix_{4})\wedge (x_{1}-ix_{4})], \\[5pt]
  && [x_{2}\wedge (x_{1}+ix_{4})\wedge (x_{1}-ix_{4})],\;  [e^{-ix_{2}\wedge x_{3}}\wedge (x_{1}+ix_{4})] \rangle\\[5pt]
 GH_{BC}^{-1}(M) &=& \langle [e^{-ix_{2}\wedge x_{3}}],\; [e^{-ix_{2}\wedge x_{3}}\wedge (x_{1}+ix_{4})\wedge (x_{1}-ix_{4})],\; \\[5pt]
  && [(x_{1}-ix_{4})\wedge x_{2}] \rangle\\[5pt]
 GH_{BC}^{-2}(M) &=&\langle [\bar\rho] \rangle \;.
\end{eqnarray*}

\end{document}